\documentclass[11pt,a4paper]{article}

\usepackage{amssymb,amsmath,amsthm}
\usepackage{braket}

\theoremstyle{plain}
\newtheorem{thm}{Theorem}[section]
\newtheorem{prop}[thm]{Proposition}
\newtheorem{lem}[thm]{Lemma}

\theoremstyle{definition}
\newtheorem{dfn}[thm]{Definition}
\newtheorem{rem}[thm]{Remark}

\numberwithin{equation}{section}

\makeatletter
\renewenvironment{proof}[1][\proofname]{\par
  \pushQED{\qed}%
  \normalfont \topsep6\p@\@plus6\p@\relax
  \trivlist
  \item[\hskip\labelsep
	\bfseries
    #1\@addpunct{.}]\ignorespaces
}{%
  \popQED\endtrivlist\@endpefalse
}
\makeatother

\DeclareMathOperator{\End}{End}

\newcommand{\bbZ}{\mathbb{Z}}
\newcommand{\bbC}{\mathbb{C}}

\newcommand{\frakS}{\mathfrak{S}}

\newcommand{\affY}{\hat{\mathbf{Y}}}
\newcommand{\finY}{\mathbf{Y}}

\newcommand{\affH}{\hat{\mathbf{H}}}
\newcommand{\finH}{\mathbf{H}}

\newcommand{\gl}{\mathfrak{gl}}
\newcommand{\fraksl}{\mathfrak{sl}}
\newcommand{\affsl}{\hat{\mathfrak{sl}}_N}

\newcommand{\cP}{\mathcal{P}}
\newcommand{\cM}{\mathcal{M}}

\newcommand{\bu}{\mathbf{u}}
\newcommand{\bv}{\mathbf{v}}
\newcommand{\bw}{\mathbf{w}}

\newcommand{\uk}{\underline{k}}
\newcommand{\uc}{\underline{c}}
\newcommand{\ulambda}{\underline{\lambda}}
\newcommand{\umu}{\underline{\mu}}

\title{Higher level Fock spaces and affine Yangian}

\author{Ryosuke Kodera}

\date{}

\makeatletter
\let\@old@@maketitle=\@maketitle
\def\@maketitle{%
\footnotetext{%
\hspace*{-1em}\hspace*{-\footnotesep}%
Research Institute for Mathematical Sciences,
Kyoto University, Kyoto 606-8502,
Japan\\
E-mail address: kodera@kurims.kyoto-u.ac.jp
}
\@old@@maketitle
}
\makeatother

\begin{document}
\maketitle

\begin{abstract}
We construct actions of the affine Yangian of type A on higher level Fock spaces by extending known actions of the Yangian of finite type A due to Uglov.
This is a degenerate analog of a result by Takemura-Uglov, which constructed actions of the quantum toroidal algebra on higher level $q$-deformed Fock spaces.
\end{abstract}

\section{Introduction}

The aim of this paper is to construct representations of the affine Yangian of type A.
The affine Yangian of type A is a two parameter deformation of the enveloping algebra of the universal central extension of a Lie algebra $\fraksl_N[u^{\pm 1},v]$ and its defining relations were first introduced by Guay \cite{MR2199856}.
In \cite{MR2199856}, Guay established a Schur-Weyl type duality between the affine Yangian and the degenerate double affine Hecke algebra.
His result gives a fully faithful embedding of the module category of the degenerate double affine Hecke algebra into the module category of the affine Yangian.
However there are many representations which can not be obtained by the functor.

Other family of representations has been constructed by geometric approach, via quiver variety by Varagnolo \cite{MR1818101}, and via affine Laumon space by Feigin-Finkelberg-Negut-Rybnikov \cite{MR2827177}.
In \cite{kodera}, the author proved that the level 1 Fock space admits an action of the affine Yangian by extending an action of the Yangian of finite type A due to Uglov \cite{MR1724950}.
This author's work uses an identification of the Fock space with the equivariant homology group of quiver variety and Varagnolo's result.

In this paper we construct actions of the affine Yangian on Fock spaces of arbitrary level by a purely algebraic method.
Given a positive integer $L$ and an $L$-tuple of integers $\uc$, we define a vector space $F(\uc)$ with a basis indexed by the set of $L$-tuples of partitions.
This space $F(\uc)$ admits an action of the affine Lie algebra $\affsl$ depending on $\uc$, and is called the level $L$ Fock space with multi-charge $\uc$.

Uglov \cite{uglov} constructed actions of the Yangian of finite type A on these Fock spaces.
His work was motivated by a study of Yangian symmetry of the spin Calogero-Sutherland model.
Takemura-Uglov \cite{MR1600311} explored a $q$-deformation of the action.
Namely they constructed a level $0$ action of the quantum affine algebra $U_q(\affsl)$ on the level $1$ $q$-deformed Fock space.
Thus the $q$-deformed Fock space admits two different types of actions of the quantum affine algebra, one is of level $1$ and the other is of level $0$.
The reason of the existence of the two types of actions is explained by Varagnolo-Vasserot \cite{MR1626481} and Saito-Takemura-Uglov \cite{MR1603798}: the $q$-deformed Fock space admits an action of the quantum toroidal algebra.
Their result was extended by Takemura-Uglov \cite{MR1710750} for higher level $q$-deformed Fock spaces.

Now it is quite natural to expect that a degenerate analog of the above result holds.
Since in 90's the definition of affine Yangian was not known, the degenerate version has been missing.
Our main result gives an affirmative answer.

\begin{thm}[Theorem \ref{thm:main}]
The actions of the affine Lie algebra $\affsl$ and the Yangian $\finY$ on the level $L$ Fock space $F(\uc)$ is glued and extended to an action of the affine Yangian $\affY$.
\end{thm}

A strategy of the proof is the same as \cite{MR1710750}.
The affine Yangian has an automorphism corresponding to the rotation of the Dynkin diagram of affine type A.
We define a bijective $\bbC$-linear map corresponding to the Dynkin automorphism on the semi-infinite wedge space and use it to construct the action of the affine Yangian.

The paper is organized as follows.
In Section \ref{sec:affine_Yangian}, we introduce affine Yangian and recall some results of Guay.
In Section \ref{sec:Fock_space}, we recall the Yangian actions on the Fock spaces constructed by Uglov.
Main ingredients are semi-infinite wedge construction of Fock space and Schur-Weyl type duality due to Drinfeld.
In Section \ref{sec:action}, we prove the main theorem.

\subsection*{Acknowledgments}
The author would like to thank Kentaro Wada for discussions on higher level Fock spaces.
He also thanks Hiraku Nakajima for discussions and encouragement.
This work was supported by JSPS KAKENHI Grant Number 25220701.

\section{Affine Yangian and Schur-Weyl type functors}\label{sec:affine_Yangian}

\subsection{Affine Yangian}

Fix an integer $N \geq 3$ throughout the paper.

\begin{dfn}
The affine Yangian $\affY = \mathbf{Y}_{\hbar,\beta}(\affsl)$ with parameters $\hbar,\beta \in \bbC$ is the $\bbC$-algebra generated by $X_{i,r}^+, X_{i,r}^-, H_{i,r}$ ($i \in \bbZ / N\bbZ$, $r \in \bbZ_{\geq 0}$) subject to the following relations:
\begin{equation}
	[H_{i,r}, H_{j,s}] = 0, \label{eq:Y1}
\end{equation}
\begin{equation}
	[X_{i,r}^{+}, X_{j,s}^{-}] = \delta_{ij} H_{i, r+s}, \label{eq:Y2}
\end{equation}
\begin{equation}
	[H_{i,0}, X_{j,s}^{\pm}] = \pm a_{ij}X_{j,s}^{\pm} \label{eq:Y3}
\end{equation}
for all $i,j$;
\begin{align}
	[H_{i,r+1}, X_{j, s}^{\pm}] - [H_{i, r}, X_{j, s+1}^{\pm}] &= \pm \dfrac{1}{2}\hbar a_{ij} (H_{i, r}X_{j, s}^{\pm} + X_{j, s}^{\pm}H_{i, r}), \label{eq:Y4} \\
	[X_{i,r+1}^{\pm}, X_{j, s}^{\pm}] - [X_{i, r}^{\pm}, X_{j, s+1}^{\pm}] &= \pm \dfrac{1}{2}\hbar a_{ij} (X_{i, r}^{\pm}X_{j, s}^{\pm} + X_{j, s}^{\pm}X_{i, r}^{\pm}) \label{eq:Y5}
\end{align}
for $(i,j) \neq (1,0), (0,1), (N-1,0), (0,N-1)$;
\begin{align}
	[H_{i, r+1}, X^{+}_{j, s}] - [H_{i, r}, X^{+}_{j, s+1}] &= (\beta - \hbar)H_{i, r}X^{+}_{j, s} - \beta X^{+}_{j, s}H_{i, r}, \label{eq:Y6} \\
	[H_{i, r+1}, X^{-}_{j, s}] - [H_{i, r}, X^{-}_{j, s+1}] &= \beta H_{i, r}X^{-}_{j, s} + (\hbar - \beta)X^{-}_{j, s}H_{i, r}, \label{eq:Y7} \\
		[X^{+}_{i, r+1}, X^{+}_{j, s}] - [X^{+}_{i, r}, X^{+}_{j, s+1}] &= (\beta - \hbar)X^{+}_{i, r}X^{+}_{j, s} - \beta X^{+}_{j, s}X^{+}_{i, r}, \label{eq:Y8} \\
	[X^{-}_{i, r+1}, X^{-}_{j, s}] - [X^{-}_{i, r}, X^{-}_{j, s+1}] &= \beta X^{-}_{i, r}X^{-}_{j, s} + (\hbar - \beta)X^{-}_{j, s}X^{-}_{i, r} \label{eq:Y9}
\end{align}
for $(i,j) = (1,0)$ and $(0,N-1)$;
\begin{align}
	[H_{i, r+1}, X^{+}_{j, s}] - [H_{i, r}, X^{+}_{j, s+1}] &= - \beta H_{i, r}X^{+}_{j, s} + (\beta - \hbar)X^{+}_{j, s}H_{i, r}, \label{eq:Y10} \\
	[H_{i, r+1}, X^{-}_{j, s}] - [H_{i, r}, X^{-}_{j, s+1}] &= (\hbar - \beta) H_{i, r}X^{-}_{j, s} + \beta X^{-}_{j, s}H_{i, r} \label{eq:Y11}
\end{align}
for $(i,j) = (0,1)$ and $(N-1,0)$;
\begin{equation}
	\sum_{w \in \frakS_{1-a_{ij}}} [X_{i,r_{w(1)}}^{\pm}, [X_{i,r_{w(2)}}^{\pm}, \dots, [X_{i,r_{w(1 - a_{ij})}}^{\pm}, X_{j,s}^{\pm}]\dots]] = 0 \label{eq:Y12}
\end{equation}
for $i \neq j$, where
\[
a_{ij} =
\begin{cases}
2  \text{ if } i=j, \\
-1 \text{ if } i=j \pm 1, \\
0  \text{ otherwise.}
\end{cases}
\]

The Yangian $\finY = \mathbf{Y}_{\hbar}(\fraksl_N)$ with a parameter $\hbar \in \bbC$ is the $\bbC$-algebra generated by $X_{i,r}^+, X_{i,r}^-, H_{i,r}$ ($i = 1, \ldots, N-1$, $r \in \bbZ_{\geq 0}$) subject to the relations (\ref{eq:Y1}), (\ref{eq:Y2}), (\ref{eq:Y3}), (\ref{eq:Y4}), (\ref{eq:Y5}), (\ref{eq:Y12}).
\end{dfn}

Let $\affY^{(1)}$ be the subalgebra of $\affY$ generated by $X_{i,r}^+, X_{i,r}^-, H_{i,r}$ ($i = 1, \ldots, N-1$, $r \in \bbZ_{\geq 0}$) and $\affY^{(2)}$ be the subalgebra of $\affY$ generated by $X_{i,0}^+, X_{i,0}^-, H_{i,0}$ ($i \in \bbZ / N\bbZ$).
Then there exist surjective algebra homomorphisms $\finY \to \affY^{(1)}$ and $U(\affsl) \to \affY^{(2)}$.
In \cite{MR2323534}, Guay proved a PBW type theorem for the affine Yangian $\affY$ when $N \geq 4$.
It implies that the two homomorphisms above are isomorphisms if $N \geq 4$.
We sometimes write $X_{i}^{\pm} = X_{i,0}^{\pm}$, $H_{i} = H_{i,0}$ and identify them with the standard Chevalley generators of $\affsl$.

Let $E_{ij}$ be the matrix unit with $(i,j)$-th entry $1$.
We regard them as elements of the Lie algebra $\gl_N$.
We put
\begin{align*}
	J(X_{i}^+) &= X_{i,1}^+ + \dfrac{\hbar}{4} \Bigg( \sum_{p=i+1}^N (E_{i,p} E_{p,i+1} + E_{p,i+1} E_{i,p}) - \sum_{p=1}^i (E_{i,p} E_{p,i+1} + E_{p,i+1} E_{i,p}) \Bigg), \\
	J(X_{i}^-) &= X_{i,1}^- + \dfrac{\hbar}{4} \Bigg( \sum_{p=i+1}^N (E_{i+1,p} E_{p,i} + E_{p,i} E_{i+1,p}) - \sum_{p=1}^i (E_{i+1,p} E_{p,i} + E_{p,i} E_{i+1,p}) \Bigg), \\
	J(H_i) &= H_{i,1} + \dfrac{\hbar}{4} \Bigg( \sum_{p=i+1}^N (E_{i,p} E_{p,i} + E_{p,i} E_{i,p}) - \sum_{p=1}^{i-1} (E_{i,p} E_{p,i} + E_{p,i} E_{i,p}) \\
	+ & \sum_{p=1}^i (E_{i+1,p} E_{p,i+1} + E_{p,i+1} E_{i+1,p}) - \sum_{p=i+2}^{N} (E_{i+1,p} E_{p,i+1} + E_{p,i+1} E_{i+1,p}) - 2H_i^2 \Bigg)
\end{align*}
for $i = 1, \ldots, N-1$.
Then $\{ X_{i}^+, X_{i}^-, H_i, J(X_i^+), J(X_i^-), J(H_i)\ (i=1, \ldots, N-1) \}$ gives another set of generators for the Yangian $\finY$.

\begin{lem}[{\cite[Lemma 3.5]{MR2199856}}]
There exists an algebra automorphism $\rho$ of $\affY$ defined by
\[
	\rho(X_{i,r}^{\pm}) = \sum_{s=0}^r \dbinom{r}{s} \left( \dfrac{\hbar}{2} \right)^{r-s} X_{i-1,s}^{\pm}, \quad \rho(H_{i,r}) = \sum_{s=0}^r \dbinom{r}{s} \left( \dfrac{\hbar}{2} \right)^{r-s} H_{i-1,s}
\]
if $i \neq 0,1$, and
\[
	\rho(X_{i,r}^{\pm}) = \sum_{s=0}^r \dbinom{r}{s} \beta^{r-s} X_{i-1,s}^{\pm}, \quad \rho(H_{i,r}) = \sum_{s=0}^r \dbinom{r}{s} \beta^{r-s} H_{i-1,s}
\]
if $i = 0,1$.
\end{lem}

\begin{prop}[{\cite[(6.44)]{MR2199856}}]\label{prop:sufficient_condition}
Let $\varphi \colon \finY \to \End_{\bbC}M$ be a $\finY$-module.
Suppose that there exists a bijective $\bbC$-linear endomorphism $T$ of $M$ satisfying:
\begin{align*}
	\varphi(\rho(X_{i,r}^{\pm})) &= T^{-1} \varphi(X_{i,r}^{\pm}) T, \\
	\varphi(\rho(H_{i,r})) &= T^{-1} \varphi(H_{i,r}) T
\end{align*}
for $i=2, \ldots, N-1$, and
\begin{align*}
	\varphi(\rho^2(X_{1,r}^{\pm})) &= T^{-2} \varphi(X_{1,r}^{\pm}) T^2, \\
	\varphi(\rho^2(H_{1,r})) &= T^{-2} \varphi(H_{1,r}) T^2.
\end{align*}
Then $\varphi$ extends to a $\affY$-module by
\begin{align*}
	\varphi(X_{0,r}^{\pm}) &= T \varphi(\rho(X_{0,r}^{\pm})) T^{-1}, \\
	\varphi(H_{0,r}) &= T \varphi(\rho(H_{0,r})) T^{-1}.
\end{align*}
\end{prop}

\subsection{Degenerate double affine Hecke algebra}

Fix an integer $n \geq 1$.

\begin{dfn}
The degenerate double affine Hecke algebra $\affH = \mathbf{H}_{t,c}(\hat{\frakS}_n)$ with parameters $t, c \in \bbC$ is the $\bbC$-algebra generated by $s_1, \ldots, s_{n-1}$, $x_1^{\pm 1}, \ldots, x_n^{\pm 1}$, $u_1, \ldots, u_n$ subject to the following relations:
\begin{gather}
	s_i^2 = 1, \ s_i s_j = s_j s_i \ (i \neq j, j \pm 1), \ s_i s_{i+1} s_i = s_{i+1} s_i s_{i+1}, \label{eq:H1} \\
	x_i x_i^{-1} = x_i^{-1} x_i = 1, \ x_i x_j = x_j x_i, \	s_i x_i = x_{i+1} s_i, \ s_i x_j = x_j s_i \ (j \neq i,i+1), \label{eq:H2} \\
	u_i u_j = u_j u_i, \ s_i u_i - u_{i+1} s_i = -c, \ s_i u_j = u_j s_i \ (j \neq i,i+1), \label{eq:H3} \\
	[u_i, x_i] = t x_i + c \left( \sum_{j<i} x_j s_{ji} + \sum_{i<j} x_i s_{ij} \right), \ [u_i, x_j] = \begin{cases}
		-c x_i s_{ij} \text{ if } i < j, \\
		-c x_j s_{ji} \text{ if } j < i, \label{eq:H4}
	\end{cases} 
\end{gather}
where $s_{ij}$ denotes the permutation of $i$ and $j$.

The degenerate affine Hecke algebra $\finH = \mathbf{H}_{c}(\frakS_n)$ with a parameter $c \in \bbC$ is the $\bbC$-algebra generated by $s_1, \ldots, s_{n-1}$, $u_1, \ldots, u_n$ subject to the relations (\ref{eq:H1}) and (\ref{eq:H3}).
\end{dfn}

It is well known that 
\[
	\affH \cong \bbC[x_1^{\pm 1}, \ldots, x_n^{\pm 1}] \otimes \bbC \frakS_n \otimes \bbC[u_1, \ldots, u_n]
\]
as a $\bbC$-vector space, hence the subalgebra of $\affH$ generated by $s_1, \ldots, s_{n-1}$, $u_1, \ldots, u_n$ is isomorphic to $\finH$, and the subalgebra generated by $s_1, \ldots, s_{n-1}$, $x_1^{\pm 1}, \ldots, x_n^{\pm 1}$ is isomorphic to $\bbC[x_1^{\pm 1}, \ldots, x_n^{\pm 1}] \rtimes \bbC \frakS_n$.
The latter is identified with the group algebra of the extended affine Weyl group $\hat\frakS_n$ of $GL_n$.

We put
\[
	y_i^{(n)} = u_i + \dfrac{c}{2} \left( \sum_{i<j} s_{ij} - \sum_{j<i} s_{ji} \right).
\]

\subsection{Schur-Weyl type functors}

Let $V$ be the vector representation of $\gl_N$ with a standard basis $v_1, \ldots, v_N$, that is, $\gl_N$ acts on $V$ by $E_{ij} v_k = \delta_{jk} v_i$.
The symmetric group $\frakS_n$ acts on $V^{\otimes n}$ via permutation of factors.
For an element $X \in \gl_N$ and $k=1, \ldots, n$, we put
\[
	(X)_k = 1^{\otimes k-1} \otimes X \otimes 1^{\otimes n-k} \in \End V^{\otimes n}.
\]

\begin{thm}[\cite{MR831053}]
Let $M$ be a right $\finH$-module and set $\hbar = c$.
Then the Yangian $\finY$ acts on $M \otimes_{\bbC \frakS_n} V^{\otimes n}$ by
\begin{align*}
	X (m \otimes \bv) &= \sum_{k=1}^n m \otimes (X)_k \bv, \\
	J(X) (m \otimes \bv) &= \sum_{k=1}^n m y_k^{(n)} \otimes (X)_k \bv
\end{align*}
for $X=X_{i}^{\pm}, H_i$ $(i=1, \ldots, N-1)$ and $m \in M, \, \bv \in V^{\otimes n}$.
\end{thm}

It is easy to see that, for a right $\hat\frakS_n$-module $M$, the Lie algebra $\fraksl_N[z^{\pm 1}]$ acts on $M \otimes_{\bbC \frakS_n} V^{\otimes n}$ by
\[
	Xz^r (m \otimes \bv) = \sum_{k=1}^n m x_k^r \otimes (X)_k \bv
\]
for $X \in \fraksl_N$ and $r \in \bbZ$.

Guay proved in \cite{MR2199856} that the actions of $\finY$ and $\affsl$ (with trivial center) is glued and extended to an action of the affine Yangian $\affY$ provided $M$ is an $\affH$-module.
We recall Guay's argument.
Let $M$ be a right $\affH$-module.
Define a bijective $\bbC$-linear endomorphism $T$ of $M \otimes_{\bbC \frakS_n} V^{\otimes n}$ by
\begin{equation}\label{eq:T}
	T (m \otimes v_{a_1} \otimes \cdots \otimes v_{a_n} ) = (m x_{1}^{-\delta_{a_1,N}} \cdots x_{n}^{-\delta_{a_n,N}}) \otimes v_{a_1 + 1} \otimes \cdots \otimes v_{a_n + 1}.
\end{equation}

\begin{lem}[{\cite[Lemma 6.2]{MR2199856}}]\label{lem:Guay}
Set $\hbar = c$ and $\beta = t/2 - Nc/4 + c/2$.
Then the bijection $T$ satisfies the condition in Proposition~\ref{prop:sufficient_condition}.
\end{lem}

\begin{thm}[{\cite[Theorem 5.4]{MR2199856}}]\label{thm:Schur-Weyl}
Let $M$ be a right $\affH$-module and set $\hbar = c$ and $\beta = t/2 - Nc/4 + c/2$.
Then the affine Yangian $\affY$ acts on $M \otimes_{\bbC \frakS_n} V^{\otimes n}$ by extending the actions of $\finY$ and $\affsl$.
\end{thm}

\section{Higher level Fock spaces}\label{sec:Fock_space}

\subsection{Semi-infinite wedge construction}

We follow \cite{MR3202706} to construct higher level Fock spaces via semi-infinite wedge spaces.
Fix an integer $L \geq 1$.
Let
\[
	V= \bigoplus_{a=1}^N \bbC v_a, \quad W= \bigoplus_{b=1}^L \bbC w_a
\]
be the vector representations of $\gl_N$ and $\gl_L$ with standard bases.
Set
\[
	U= \bbC[z^{\pm1}] \otimes W \otimes V
\]
and denote an element $z^m \otimes w_b \otimes v_a$ by $z^m w_b v_a$.
Define $\bu_k \in U$ ($k \in \bbZ$) by
\[
	z^m w_b v_a = \bu_{a - N(b + Lm)}.
\]
Then $\{ \bu_k \ (k \in \bbZ) \}$ forms a basis of $U$.
The wedge space	$\bigwedge^n U$ has a basis $\{ \bu_{k_1} \wedge \cdots \wedge \bu_{k_n} \ (k_1 > \cdots > k_n) \}$.
For a sequence $\uk = (k_1, \ldots, k_n)$, we put
\[
	\bu_{\uk} = \bu_{k_1} \wedge \cdots \wedge \bu_{k_n}.
\]
Fix an integer $M$.
For $n < m$, we define a $\bbC$-linear map $\bigwedge^n U \to \bigwedge^m U$ by
\[
	v \mapsto v \wedge \bu_{M-n} \wedge \bu_{M-n-1} \wedge \cdots \wedge \bu_{M-m+1}.
\]
These form an inductive system.
We define
\[
	F_M = \varinjlim_n \bigwedge^n U.
\]
Put
\[
	\cM= \left\{ \uk = (k_1, k_2, \ldots)\ \middle|
	\begin{array}{l}
		k_i \in \bbZ \text{ and } k_i > k_{i+1} \text{ for all }i, \\
		k_i = M-i+1 \text{ for all but finitely many }i
	\end{array}\right\}.
\]
Then $F_M$ has a basis consisting of vectors
\[
	\bu_{\uk} = \bu_{k_1} \wedge \bu_{k_2} \wedge \cdots
\]
for $\uk=(k_1, k_2, \ldots) \in \cM$.
For an integer $m \in \bbZ$, put
\[
	\ket{m} = \bu_m \wedge \bu_{m-1} \wedge \cdots.
\]
By the definition, every element of $F_M$ is of the form $v \wedge \ket{M-n}$ for some $n \geq 0$ and $v \in \bigwedge^n U$.

Let
\[
	\cP= \left\{ \lambda=(\lambda_1, \lambda_2, \ldots)\ \middle|
	\begin{array}{l}
		\lambda_i \in \bbZ_{\geq 0} \text{ and } \lambda_i \geq \lambda_{i+1} \text{ for all }i, \\
		\lambda_i = 0 \text{ for all but finitely many } i
	\end{array}\right\}
\]
be the set of partitions.
We associate a partition $\lambda \in \cP$ with $\uk \in \cM$ by
\[
	k_i = M + \lambda_i - i + 1
\]
for every $i$.
This gives a bijection between $\cM$ and $\cP$.
We write $\ket{\lambda,M} = \bu_{\uk}$ via this correspondence.

We decompose $F_M$ using a variant of $L$-quotient for partitions.
Take a basis element $\bu_{\uk}$ ($\uk \in \cM$) of $F_M$.
Recall that
\[
	\bu_{k_i} = z^{m_i} w_{b_i} v_{a_i},
\]
where
\[
	k_i = a_i - N(b_i + Lm_i).
\]
For each $s = 1, \ldots, L$, define $j_1^{(s)} < j_2^{(s)} < \cdots$ to be the indices satisfying
\[
	s= b_{j_1^{(s)}} = b_{j_2^{(s)}} = \cdots.
\]
Put $a_i^{(s)} = a_{j_i^{(s)}}$, $m_i^{(s)} = m_{j_i^{(s)}}$, $k_i^{(s)} = a_i^{(s)} - N m_i^{(s)}$.
Then it is easy to see that there exists $c_s \in \bbZ$ for each $s=1, \ldots, L$ such that
\[
	k_i^{(s)} = c_s - i + 1
\]
for all but finitely many $i$.
These $\{ c_s \}$ satisfy $c_1 + \cdots + c_L = M$.
Define an $L$-partition $\ulambda=(\lambda^{(1)}, \ldots, \lambda^{(L)}) \in \cP^L$ by
\[
	k_i^{(s)} = c_s + \lambda_i^{(s)} - i + 1.
\]
The assignment
\[
	\cP \ \to \ \cP^L \times \{ \uc = (c_1, \ldots, c_L) \in \bbZ^L \mid c_1 + \cdots + c_L = M \}, \quad \lambda \mapsto (\ulambda, \uc)
\]
gives a bijection.
See \cite[Remark 4.2 (ii)]{MR1768086} for a relation to the operations of $L$-quotient and $L$-core for partitions.
We write $\ket{\ulambda, \uc} = \ket{\lambda,M} = \bu_{\uk}$ via this correspondence.
Set
\[
	F(\uc) = \bigoplus_{\ulambda \in \cP^L} \bbC \ket{\ulambda, \uc}
\]
for each $\uc$.
This gives a decomposition
\[
	F_M = \bigoplus_{\substack{\uc \in \bbZ^L\\ c_1 + \cdots + c_L = M}} F(\uc).
\]

\subsection{Affine Lie algebra action}

A natural action of $\fraksl_N[z^{\pm 1}]$ on $\bbC[z^{\pm 1}] \otimes V$ is given by
\[
	Xz^r (z^m \otimes v) = z^{m+r} \otimes Xv.
\]
This yields a level $L$ action of the affine Lie algebra $\affsl$ on $F_M$.
The action preserves each component $F(\uc)$ and is described in terms of combinatorics of $L$-partitions.
We identify each partition $\lambda \in \cP$ with its Young diagram
\[
	\lambda = \{ (x,y) \in (\bbZ_{>0})^2 \mid x=1, \ldots, l(\lambda),\ y=1, \ldots, \lambda_x \}.
\]
Here $l(\lambda)$ denotes the length of $\lambda$.
Take $(\ulambda, \uc) \in \cP^L \times \{ \uc \in \bbZ^L \mid c_1 + \cdots + c_L = M \}$ and fix $s = 1, \ldots, L$.
We call $(x,y) \in (\bbZ_{>0})^2$ an $i$-cell if
\[
	c_s + y - x \equiv i \mod N.
\]
An $i$-cell $(x,y)$ is said to be removable if $(x,y) \in \lambda^{(s)}$ and $\lambda^{(s)} \setminus \{(x,y)\}$ is a partition.
It is said to be addable if $(x,y) \notin \lambda^{(s)}$ and $\lambda^{(s)} \cup \{(x,y)\}$ is a partition.

\begin{prop}[{\cite[Proposition 4]{MR3202706}}]
The action of $\affsl$ on $F(\uc) = \bigoplus_{\ulambda \in \cP^L} \bbC \ket{\ulambda,\uc}$ is given by
\[
	X_i^+ \ket{\ulambda,\uc} = \sum_{\umu} \ket{\umu,\uc},
\]
where the sum is over all $L$-partitions $\umu$ obtained from $\ulambda$ by removing $i$-cell,
\[
	X_i^- \ket{\ulambda,\uc} = \sum_{\umu} \ket{\umu,\uc},
\]
where the sum is over all $L$-partitions $\umu$ obtained from $\ulambda$ by adding $i$-cell, and
\[
	H_i \ket{\ulambda,\uc} = \Big( \#\{ \text{\rm addable $i$-cells in $\ulambda$} \} - \#\{ \text{\rm removable $i$-cells in $\ulambda$} \} \Big) \ket{\ulambda,\uc}.
\]
\end{prop}

We call the $\affsl$-module $F(\uc)$ the level $L$ Fock space with multi-charge $\uc$.

\subsection{Yangian action}

We construct an action of the Yangian $\finY$ on each $F(\uc)$ following \cite{uglov}.

We introduce some operators acting on the space $\bbC[z_1^{\pm 1}, \ldots, z_n^{\pm 1}] \otimes W^{\otimes n}$.
Let $K_{ij}$ be the permutation of $z_i$ and $z_j$.
Let $P_{ij}$ be the permutation of the $i$-th and $j$-th factors of $W^{\otimes n}$.
Put
\[
	r_{ij} = \dfrac{1}{2} \sum_{a=1}^L (E_{aa})_i (E_{aa})_j + \sum_{1 \leq a < b \leq L} (E_{ab})_i (E_{ba})_j.
\]
Fix an $L$-tuple of complex parameters $\nu(1), \ldots, \nu(L)$ and define a matrix $\nu$ by $\nu = \sum_{a=1}^L \nu(a) E_{aa}$.
Put
\[
	\nu_i = 1^{\otimes i-1} \otimes \nu \otimes 1^{\otimes n-i} \in \End W^{\otimes n}.
\]
A matrix Dunkl-Cherednik operator $d_i^{(n)}$ for $i=1, \ldots, n$ is defined by
\begin{equation*}
	\begin{split}
		& d_i^{(n)} = t z_i \dfrac{\partial}{\partial z_i} \\
		& \quad -  c \Bigg( \nu_i + \dfrac{n}{2L} - \dfrac{1}{2}  + \sum_{j < i} \Big( \dfrac{z_i}{z_i-z_j} (1 - K_{ji})P_{ji} - r_{ji} \Big ) + \sum_{i < j} \Big( \dfrac{z_j}{z_i-z_j} (1 - K_{ij})P_{ij} + r_{ij} \Big ) \Bigg),
	\end{split}
\end{equation*}
which acts on $\bbC[z_1^{\pm 1}, \ldots, z_n^{\pm 1}] \otimes W^{\otimes n}$.

\begin{prop}[{\cite[Proposition 2.8]{uglov}}]
The assignment 
\[
	s_i \mapsto - K_{i,i+1} P_{i,i+1}, \quad x_i \mapsto z_i, \quad	u_i \mapsto - d_i^{(n)}
\]
gives a right action of the degenerate double affine Hecke algebra $\affH$ on $\bbC[z_1^{\pm 1}, \ldots, z_n^{\pm 1}] \otimes W^{\otimes n}$.
\end{prop}

\begin{rem}
A formula for the action given in \cite{uglov} is slightly different.
Note that there exists an anti-automorphism of $\affH$ defined by
\[
	s_i \mapsto s_i, \quad x_i \mapsto x_i^{-1}, \quad u_i \mapsto u_i.
\]
Thus left and right modules of $\affH$ are interchanged via the anti-automorphism.
\end{rem}

The following are easy to prove.

\begin{lem}\label{lem:standard}
\begin{enumerate}
\item
Suppose $i < j$.
Then for $P= z_1^{m_1} \cdots z_n^{m_n}$, we have
\begin{equation*}
	\dfrac{z_j}{z_i-z_j} (1 - K_{ij})P = \begin{cases}
	0 \text{ if } m_i = m_j, \\
	\\
	- \left(\displaystyle\sum_{r=0}^{m_j-m_i-1} z_i^r z_j^{-r}\right)P \text{ if } m_i < m_j.
	\end{cases}
\end{equation*}
\item
For $\bw = w_{b_1} \otimes \cdots \otimes w_{b_n}$, we have
\begin{equation*}
	r_{ij} \bw = \begin{cases}
		 (1/2) \bw \text{ if } b_i = b_j, \\
		 P_{ij} \bw \text{ if } b_i > b_j, \\
		 0 \text{ if } b_i < b_j.
	\end{cases}
\end{equation*}
\end{enumerate}
\end{lem}

We have the following observation (See \cite[Proposition 3]{MR3202706}):
\[
	\bigwedge^n U \cong \left( \bbC[z_1^{\pm 1}, \ldots, z_n^{\pm 1}] \otimes W^{\otimes n} \right) \otimes_{\bbC \frakS_n} V^{\otimes n}
\]
as a $\bbC$-vector space.
Thus we identify them and apply Guay's result (Theorem~\ref{thm:Schur-Weyl}) to conclude that the affine Yangian $\affY$ acts on the finite wedge space $\bigwedge^n U$.
Our goal is to extend this action to the limit $n \to \infty$.
In the remainder of this section, we explain how to extend the action of the Yangian $\finY$ to the limit.

Let us introduce the degree on $\bigwedge^n U$ and $F_M$.
Take an element $\bu_{\uk} \in F_M$ ($\uk \in \cM$), where
\[
	\bu_{k_i} = z^{m_i} w_{b_i} v_{a_i}.
\]
Denote the integer $m_i$ by $m_i(\uk)$.
Put $m_i^0$ to be the integer $m_i(\uk)$ for $\uk = (M,M-1, \ldots)$.
Note that $m_i(\uk) \leq m_i^0$ holds for all $i$ and the equality holds for all but finitely many $i$.
Then we define the degree of $\bu_{\uk}$ by
\[
	\deg \bu_{\uk} = \sum_{i=1}^{\infty} (m_i^0 - m_i(\uk)).
\]
The right-hand side is a finite sum.
Let $F_M^d$ be the degree $d$ component of $F_M$.
We have
\[
	F_M = \bigoplus_{d \geq 0} F_M^d.
\]
We also define the degree on $\bigwedge^n U$ similarly:
\[
	\deg \bu_{\uk} = \sum_{i=1}^{n} (m_i^0 - m_i(\uk))
\]
for $\uk=(k_1, \ldots, k_n) \in \bbZ^n$ such that $k_1 > \cdots > k_n$.
We define a subspace $V_{M,n}$ of $\bigwedge^n U$ by
\[
	V_{M,n} = \bigoplus_{\substack{\uk=(k_1, \ldots, k_n) \in \bbZ^n\\ k_1 > \cdots > k_n\\ m_i(\uk) \leq m_i^0}} \bbC \bu_{\uk} 
\]
and let $V_{M,n}^d$ be the degree $d$ component.
Then we have
\[
	V_{M,n} = \bigoplus_{d \geq 0} V_{M,n}^d.
\]

\begin{prop}[{\cite[Proposition 3.1]{uglov}}]
Each $V_{M,n}^d$ is invariant under the action of $\finY$.
\end{prop}

\begin{rem}
The action of $\affsl$ does not preserve the degree.
Hence $V_{M,n}^d$ is not invariant under the action of $\affY$.
\end{rem}

Let $s \in \{ 0,1, \ldots, NL-1 \}$ be the number defined by
\[
	s \equiv M \mod NL.
\]

\begin{prop}[{\cite[Proposition 3.3 and 3.8]{uglov}}]\label{prop:intertwine}
\begin{enumerate}
\item
If $l \geq d$ then
\[
	V_{M,s+lNL}^d \to F_M^d, \quad v \mapsto v \wedge \ket{M-(s+lNL)}
\]
is an isomorphism of $\bbC$-vector spaces.
In particular, if $l' \geq l \geq d$ then
\[
	V_{M,s+lNL}^d \to V_{M,s+l'NL}^d, \quad v \mapsto v \wedge \bu_{M-(s+lNL)} \wedge \cdots \wedge \bu_{M-(s+l'NL)+1} 
\]
is an isomorphism of $\bbC$-vector spaces.
\item
If $l' \geq l \geq d$ then 
\[
	V_{M,s+lNL}^d \to V_{M,s+l'NL}^d
\]
is an isomorphism of $\finY$-modules.
\end{enumerate}
\end{prop}

\begin{rem}
Uglov \cite{uglov} considered actions of the bigger algebra $\finY(\gl_N)$.
To make the map in Proposition~\ref{prop:intertwine} intertwine the $\finY(\gl_N)$-actions, one needs to correct the $\finY(\gl_N)$-actions by some twist, which does not change the $\finY$-module structure.
\end{rem}

By Proposition~\ref{prop:intertwine}, an action of $\finY$ on $F_M^d$ is well-defined through the actions on $V_{M,s+lNL}^d$ which are compatible for all $l \geq d$; for $X \in \finY$ and $v \in V_{M,s+lNL}^d$, the action is given by
\begin{equation}
	X(v \wedge \ket{M-(s+lNL)}) = Xv \wedge \ket{M-(s+lNL)}. \label{eq:Yangian_action}
\end{equation}
This action preserves each component $F(\uc)$ since it changes only finitely many factors for each element.
 
\section{Affine Yangian actions on higher level Fock spaces}\label{sec:action}

\subsection{Main Theorem}

To construct an action of the affine Yangian $\affY$ on $F(\uc)$, we define a map $T_{\infty}$ satisfying the condition in Proposition~\ref{prop:sufficient_condition}.
Such $T_{\infty}$ is given by Takemura-Uglov \cite{MR1710750} in the quantum toroidal case and we follow it.

Consider the map $T$ given by (\ref{eq:T}) for our $\affY$-module $\left( \bbC[z_1^{\pm 1}, \ldots, z_n^{\pm 1}] \otimes W^{\otimes n} \right) \otimes_{\bbC \frakS_n} V^{\otimes n}$.
It is given by
\[
	T(z^{m_1}w_{b_1} v_{a_1} \wedge \cdots \wedge z^{m_n} w_{b_n} v_{a_n}) =  z^{m_1-\delta_{a_1,N}} w_{b_1} v_{a_1+1} \wedge \cdots \wedge z^{m_n-\delta_{a_n,N}} w_{b_n} v_{a_n+1}.
\]
We can rewrite this as follows:
\[
	T(\bu_{k_1} \wedge \cdots \wedge \bu_{k_n}) = \bu_{k_1'} \wedge \cdots \wedge \bu_{k_n'},
\]
where
\[
	k_i' = \begin{cases}
	k_i + 1 \text{ if } k_i \not\equiv 0 \mod NL, \\
	k_i + 1 + N(L-1) \text{ if } k_i \equiv 0 \mod NL.
	\end{cases}
\]
Fix $l \geq d$.
Since $s$ is defined to satisfy $M \equiv s \mod NL$, there exists a unique $m \in \bbZ$ satisfying
\[
	M - (s + lNL) = -mNL.
\]
Then by the definition every element of $F_M^d$ is written as $v \wedge \ket{-mNL}$ for some $v \in V_{M,s+lNL}^d$.
Following \cite[Lemma 6.9]{MR1710750}, we introduce elements $v_{L,N}$ and $v_{L,N-1}$ by
\begin{align*}
	v_{L,N} &= z^m w_1 v_N \wedge z^m w_2 v_N \wedge \cdots \wedge z^m w_L v_N \\
	&= \bu_{-mNL} \wedge \bu_{-mNL-N} \wedge \cdots \wedge \bu_{-mNL-(L-1)N}\end{align*}
and
\begin{align*}
	& v_{L,N-1} = z^m w_1 v_N \wedge z^m w_1 v_{N-1} \wedge z^m w_2 v_N \wedge z^m w_2 v_{N-1} \wedge \cdots \wedge z^m w_L v_N \wedge z^m w_L v_{N-1} \\
	&= \bu_{-mNL} \wedge \bu_{-mNL-1} \wedge \bu_{-mNL-N} \wedge \bu_{-mNL-N-1} \wedge \cdots \wedge \bu_{-mNL-(L-1)N} \wedge \bu_{-mNL-(L-1)N-1}.
\end{align*}
Note that
\begin{align*}
	T(v_{L,N}) &= z^{m-1} w_1 v_1 \wedge z^{m-1} w_2 v_1 \wedge \cdots \wedge z^{m-1} w_L v_1 \\
	&= \bu_{-mNL+(L-1)N+1} \wedge \bu_{-mNL+(L-2)N+1} \wedge \cdots \wedge \bu_{-mNL+1}
\end{align*}
and
\begin{align*}
	& T(v_{L,N-1}) = z^{m-1} w_1 v_1 \wedge z^m w_1 v_N \wedge z^{m-1} w_2 v_1 \wedge z^m w_2 v_N \wedge \cdots \wedge z^{m-1} w_L v_1 \wedge z^m w_L v_N \\
	&= \bu_{-mNL+(L-1)N+1} \wedge \bu_{-mNL} \wedge \bu_{-mNL+(L-2)N+1} \wedge \bu_{-mNL-N} \wedge \cdots \wedge \bu_{-mNL+1} \wedge \bu_{-mNL-(L-1)N} \\
	&= (-1)^{1+2+\cdots+(L-1)} T(v_{L,N}) \wedge v_{L,N}.
\end{align*}
Define a bijective $\bbC$-linear map $T_{\infty} \colon F_M \to F_{M+L}$ by
\begin{equation}
	T_{\infty}(v \wedge \ket{-mNL}) = T(v \wedge v_{L,N}) \wedge \ket{-mNL}. \label{eq:T_infty}
\end{equation}
In the case $L=1$, this is simply written as
\[
	T_{\infty}(\bu_{k_1} \wedge \bu_{k_2} \wedge \cdots) = \bu_{k_1 + 1} \wedge \bu_{k_2 + 1} \wedge \cdots.
\]

\begin{lem}[cf.\ {\cite[Lemma 6.9]{MR1710750}}]\label{lem:key}
We have
\[
	TX(v \wedge v_{L,N}) \wedge \ket{-mNL} = T(Xv \wedge v_{L,N}) \wedge \ket{-mNL}
\]
for $X= X_{i,r}^{\pm}, H_{i,r}$ $(i=1, \ldots, N-2)$ and
\[
	T^2X(v \wedge v_{L,N-1}) \wedge \ket{-mNL} = T^2(Xv \wedge v_{L,N-1}) \wedge \ket{-mNL}
\]
for $X= X_{N-1,r}^{\pm}, H_{N-1,r}$.
\end{lem}

We give a proof of Lemma~\ref{lem:key} in the next subsection.

\begin{prop}\label{prop:cyclic}
We have
\begin{equation*}
	T_{\infty}^{-1} X T_{\infty} (v \wedge \ket{-mNL}) = \rho(X) (v \wedge \ket{-mNL})
\end{equation*}
for $X= X_{i,r}^{\pm}, H_{i,r}$ $(i=2, \ldots, N-1)$ and
\begin{equation*}
	T_{\infty}^{-2} X T_{\infty}^2 (v \wedge \ket{-mNL}) = \rho^2(X) (v \wedge \ket{-mNL})
\end{equation*}
for $X= X_{1,r}^{\pm}, H_{1,r}$.
\end{prop}

\begin{proof}
We have
\begin{equation*}
	T_{\infty} (v \wedge \ket{-mNL}) = T(v \wedge v_{L,N}) \wedge \ket{-mNL}
\end{equation*}
by the definition (\ref{eq:T_infty}) of $T_{\infty}$.
The action of $X \in \finY$ on the right-hand side is given by
\begin{equation*}
	X(T(v \wedge v_{L,N}) \wedge \ket{-mNL}) = XT(v \wedge v_{L,N}) \wedge \ket{-mNL}
\end{equation*}
by its definition due to (\ref{eq:Yangian_action}).
We have
\begin{equation*}
	XT(v \wedge v_{L,N})= T \rho(X)(v \wedge v_{L,N})
\end{equation*}
by Lemma~\ref{lem:Guay}.
We use Lemma~\ref{lem:key} for $X= X_{i,r}^{\pm}, H_{i,r}$ $(i=2, \ldots, N-1)$ to obtain
\begin{equation*}
	T \rho(X)(v \wedge v_{L,N}) \wedge \ket{-mNL} = T (\rho(X)v \wedge v_{L,N}) \wedge \ket{-mNL}.
\end{equation*}
Hence we have
\begin{equation*}
	\begin{split}
		X T_{\infty} (v \wedge \ket{-mNL}) &= T (\rho(X)v \wedge v_{L,N} ) \wedge \ket{-mNL} \\
		&= T_{\infty} (\rho(X)v \wedge \ket{-mNL}) \\
		&= T_{\infty} \rho(X) (v \wedge \ket{-mNL}).
	\end{split}
\end{equation*}
By a similar argument, we have for $X= X_{N-1,r}^{\pm}, H_{N-1,r}$ that
\begin{equation*}
	\begin{split}
		X T_{\infty}^2 (v \wedge \ket{-mNL}) &= X T_{\infty} (T(v \wedge v_{L,N}) \wedge \ket{-mNL}) \\
		&= X (T^2(v \wedge v_{L,N}) \wedge T(v_{L,N}) \wedge \ket{-mNL}) \\
		&= (-1)^{L(L-1)/2} X (T^2(v \wedge v_{L,N-1}) \wedge \ket{-mNL}) \\
		&= (-1)^{L(L-1)/2} X T^2(v \wedge v_{L,N-1}) \wedge \ket{-mNL} \\
		&= (-1)^{L(L-1)/2} T^2 \rho^2(X) (v \wedge v_{L,N-1}) \wedge \ket{-mNL} \\
		&= (-1)^{L(L-1)/2} T^2 (\rho^2(X)v \wedge v_{L,N-1}) \wedge \ket{-mNL} \\
		&= T_{\infty}^2 (\rho^2(X)v \wedge \ket{-mNL}) \\
		&= T_{\infty}^2 \rho^2(X) (v \wedge \ket{-mNL}).
	\end{split}
\end{equation*}
\end{proof}

By Proposition~\ref{prop:sufficient_condition} and \ref{prop:cyclic}, we conclude that the affine Yangian $\affY$ acts on $F_M$.
The action preserves each component $F(\uc)$ since those of $\finY$ and $\affsl$ preserve it.
We obtain the main result of this paper.

\begin{thm}\label{thm:main}
The actions of the affine Lie algebra $\affsl$ and the Yangian $\finY$ on the level $L$ Fock space $F(\uc)$ is glued and extended to an action of the affine Yangian $\affY$.
\end{thm}

\subsection{Proof of Lemma~\ref{lem:key}}

For $P \in \bbC[z_1^{\pm 1}, \ldots, z_n^{\pm 1}]$, $\bw \in W^{\otimes n}$, and $\bv \in V^{\otimes n}$, let $\bigwedge P \otimes \bw \otimes \bv$ denote the image of $P \otimes \bw \otimes \bv$ in $\bigwedge^n U \cong (\bbC[z_1^{\pm 1}, \ldots, z_n^{\pm 1}] \otimes W^{\otimes n}) \otimes_{\bbC \frakS_n} V^{\otimes n}$.

We recall the notation; $M$ is a fixed integer and $s \in \{ 0,1, \ldots, NL-1 \}$ is uniquely determined by
\[
	M \equiv s \mod NL.
\]
Fix integers $l \geq d \geq 0$ and put $n = s+lNL$.
An integer $m$ is uniquely determined by
\[
	M - (s+lNL) = -mNL.
\]

Before starting a proof, we provide two lemmas.

\begin{lem}\label{lem:vanish}
If $v \in \bigwedge^n U$ has a factor $z^m w_b v_a$ with $a \leq N-1$ then we have
\[
	T(v) \wedge \ket{-mNL} = 0.
\]
If $v \in \bigwedge^n U$ has a factor $z^m w_b v_a$ with $a \leq N-2$ then we have
\[
	T^2(v) \wedge \ket{-mNL} = 0.
\]
\end{lem}

\begin{proof}
Write $\bu_k = z^m w_b v_a$ so that
\[
	k = a - N(b + Lm) = -mNL + a - bN.
\]
We see that $T(\bu_k) = \bu_{k+1}$ and $k+1 \leq -mNL$ if $a \leq N-1$.
This implies the first assertion.

Similarly if $a \leq N-2$ then we have $T^2(\bu_k) = \bu_{k+2}$ and $k+2 \leq -mNL$.
This implies the second assertion.
\end{proof}

Next lemma is proved in \cite[Proof of Proposition 5 (ii)]{MR1600311}.

\begin{lem}\label{lem:degree}
Fix $k \geq 1$ and let $P$ be an element of
\[
	\text{\rm span}\{ z_1^{m_1} \cdots z_{n+k}^{m_{n+k}} \mid m_1, \ldots, m_{n+k} \leq m, \ \#\{ i \mid m_i = m \} < k \}.
\]
Then we have
\[
	\bigwedge P \otimes \bw \otimes \bv \in \bigoplus_{l' > l} V_{M,n+k}^{l'}
\]
for $\bw \in W^{\otimes n+k}$ and $\bv \in V^{\otimes n+k}$.
\end{lem}

\begin{proof}
We may assume $m_1 \leq \cdots \leq m_{n+k}$.
Then by the assumption on $P$, we have $m_{n+1} < m$.
Suppose $\bigwedge P \otimes \bw \otimes \bv \neq 0$.
Then we see that $m > m_{n+1} > m_{n+1-NL} > \cdots > m_{n+1-lNL}$.
This implies
\[
	\deg \bigwedge P \otimes \bw \otimes \bv \geq l+1.
\]
\end{proof}

Let us start a proof of Lemma~\ref{lem:key}.
It is enough to show the equalities for $X= X_{i,0}^{\pm}, H_{i,0}, X_{i,1}^-$ ($i=1, \ldots, N-1$) since these elements generate $\finY$.

\begin{lem}
We have
\[
	X(v \wedge v_{L,N}) = Xv \wedge v_{L,N}
\]
for $X= X_i^{\pm}, H_i$ $(i=1, \ldots, N-2)$, and
\[
	X(v \wedge v_{L,N-1}) = Xv \wedge v_{L,N-1}
\]
for $X= X_{N-1}^{\pm}, H_{N-1}$.
\end{lem}

\begin{proof}
We have $X v_{L,N} = 0$ for $X= X_{i}^{\pm}, H_{i}$ ($i=1, \ldots, N-2$) since factors of $v_{L,N}$ concerning the $\fraksl_N$-action is only $v_N$.
This implies $X(v \wedge v_{L,N}) = Xv \wedge v_{L,N}$.

We have
\[
	X_{N-1}^+ v_N \wedge v_{N-1} = v_{N-1} \wedge v_{N-1} = 0
\]
and 
\[
	v_N \wedge X_{N-1}^+ v_{N-1} = 0.
\]
Hence we have $X_{N-1}^+ v_{L,N-1} = 0$.
Similarly we have $X_{N-1}^- v_{L,N-1} = 0$.
We have	$H_{N-1} v_{L,N-1} =  0$ since $H_{N-1} (v_N \wedge v_{N-1}) = 0$.
Therefore we have $X(v \wedge v_{L,N-1}) = Xv \wedge v_{L,N-1}$ for $X= X_{N-1}^{\pm}, H_{N-1}$.
\end{proof}

It remains to prove
\begin{equation}
	T X_{i,1}^{-} (v \wedge v_{L,N}) \wedge \ket{-mNL} = T (X_{i,1}^{-} v \wedge v_{L,N}) \wedge \ket{-mNL} \label{eq:remain1}
\end{equation}
for $i=1, \ldots, N-2$, and
\begin{equation}
	T^2 X_{N-1,1}^{-} (v \wedge v_{L,N-1}) \wedge \ket{-mNL} = T^2 (X_{N-1,1}^{-} v \wedge v_{L,N-1}) \wedge \ket{-mNL} \label{eq:remain2}.
\end{equation}

First we prove (\ref{eq:remain1}).
Recall that
\begin{equation*}
	X_{i,1}^{-} = J(X_i^-) - \dfrac{\hbar}{4} \omega_i,
\end{equation*}
where
\begin{equation*}
	\omega_i = \sum_{p=i+1}^N (E_{i+1,p} E_{p,i} + E_{p,i} E_{i+1,p}) - \sum_{p=1}^i (E_{i+1,p} E_{p,i} + E_{p,i} E_{i+1,p}).
\end{equation*}
Let $\Delta$ be the coproduct on $\gl_N$ and put
\begin{equation*}
	\begin{split}
		\omega_i' &= \Delta(\omega_i) - \omega_i \otimes 1 \\
		&= 1 \otimes \omega_i + 2 \left( \sum_{p=i+1}^N (E_{i+1,p} \otimes E_{p,i} + E_{p,i} \otimes E_{i+1,p}) - \sum_{p=1}^i (E_{i+1,p} \otimes E_{p,i} + E_{p,i} \otimes E_{i+1,p}) \right).
	\end{split}
\end{equation*}
Then we have
\begin{equation*}
	X_{i,1}^- (v \wedge v_{L,N}) = J(X_i^-) (v \wedge v_{L,N}) - \dfrac{\hbar}{4}(\omega_i v \wedge v_{L,N}) - \dfrac{\hbar}{4} \omega_i' (v \wedge v_{L,N}).
\end{equation*}
Therefore it is enough to show
\begin{equation}
	T J(X_i^-) (v \wedge v_{L,N}) \wedge \ket{-mNL} = T (J(X_i^-) v \wedge v_{L,N}) \wedge \ket{-mNL} \label{eq:J}
\end{equation}
and
\begin{equation}
	T \omega_i' (v \wedge v_{L,N}) \wedge \ket{-mNL} = 0. \label{eq:omega}
\end{equation}
We show (\ref{eq:J}).
Write $v \wedge v_{L,N} = \bigwedge P \otimes \bw \otimes \bv$, where
\begin{align*}
	P &= z_1^{m_1} \cdots z_n^{m_n} (z_{n+1} \cdots z_{n+L})^m, \\
	\bw &= w_{b_1} \otimes \cdots \otimes w_{b_n} \otimes w_1 \otimes \cdots \otimes w_L, \\
	\bv &= v_{a_1} \otimes \cdots \otimes v_{a_n} \otimes v_N^{\otimes L}.
\end{align*}
We have $(X_i^-)_k \bv = 0$ for $k > n$.
This implies
\[
	J(X_i^-) (v \wedge v_{L,N}) = \sum_{k=1}^n \bigwedge (P \otimes \bw) y_k^{(n+L)} \otimes (X_i^-)_k \bv.
\]
Then (\ref{eq:J}) follows from
\begin{equation}
	T \left( \bigwedge (P \otimes \bw)(y_k^{(n+L)} - y_k^{(n)}) \otimes (X_i^-)_k \bv \right) \wedge \ket{-mNL}. \label{eq:diff}
\end{equation}
Since we have
\[
	y_k^{(n+L)} - y_k^{(n)} = -(d_k^{(n+L)} - d_k^{(n)}) + \dfrac{\hbar}{2} \sum_{j=1}^L s_{k,n+j},
\]
the equality (\ref{eq:diff}) follows from the next lemma.

\begin{lem}\label{lem:Jvanish}
\begin{enumerate}
Suppose $i=1, \ldots, N-2$ and $k=1, \ldots, n$.
\item
We have
\[
	\bigwedge (d_k^{(n+L)} - d_k^{(n)})(P \otimes \bw) \otimes (X_i^-)_k \bv = 0.
\]
\item
We have
\[
	T \left( \bigwedge (P \otimes \bw)s_{k,n+j} \otimes (X_i^-)_k \bv \right) \wedge \ket{-mNL} = 0.
\]
\end{enumerate}
\end{lem}

\begin{proof}
We prove (i).
By the definition, we have
\[
	d_k^{(n+L)} - d_k^{(n)} = \hbar \left( \dfrac{1}{2} + \sum_{j=1}^L \left( \dfrac{z_{n+j}}{z_k - z_{n+j}}(1 - K_{k,n+j})P_{k,n+j} + r_{k,n+j} \right) \right).
\]
By Lemma~\ref{lem:standard}, we have
\begin{equation*}
	\dfrac{z_{n+j}}{z_k - z_{n+j}}(1 - K_{k,n+j}) P = - (P + P'),
\end{equation*}
where
\[
	P'= \sum_{r=1}^{m-m_k-1} z_k^r z_{n+j}^{-r} P,
\]
and
\begin{equation*}
	\sum_{j=1}^L r_{k,n+j} \bw = \dfrac{1}{2}\bw + \sum_{\substack{1 \leq j \leq L\\ j < b_k}} P_{k,n+j}\bw.
\end{equation*}
Hence we have
\[
	(d_k^{(n+L)} - d_k^{(n)})(P \otimes \bw) = \hbar \left( -\sum_{\substack{1 \leq j \leq L\\ j > b_k}} P \otimes P_{k,n+j}\bw - \sum_{j=1}^L P' \otimes P_{k,n+j}\bw \right).
\]
Provided $j > b_k$, we see that
\[
	\bigwedge P \otimes P_{k,n+j}\bw \otimes (X_i^-)_k \bv = 0
\]
since $P \otimes P_{k,n+j}\bw \otimes (X_i^-)_k \bv$ has $z^m w_{b_k} v_N$ as its $(n+b_k)$-th and $(n+j)$-th factors.
Moreover we have
\begin{equation}
	\bigwedge P' \otimes P_{k,n+j}\bw \otimes (X_i^-)_k \bv = 0. \label{eq:lower_degree}
\end{equation}
Indeed, if the left-hand side is nonzero then it has degree $d$.
However Lemma~\ref{lem:degree} implies that its degree is greater than $l$, which is contradiction.

We prove (ii).
We may assume $a_k=i$, since $(X_i^-)_k \bv = 0$ unless $a_k=i$.
Then $(P \otimes \bw)s_{k,n+j} \otimes (X_i^-)_k \bv$ has $z^m w_j v_{i+1}$ as its $k$-th factor.
By the assumption we have $i+1 \leq N-1$.
Then Lemma~\ref{lem:vanish} implies the assertion.
\end{proof}

We show (\ref{eq:omega}).
In $\omega_i'(v \wedge v_{L,N})$, all terms but concerning
\[
	E_{p,i}v \wedge E_{i+1,p}v_{L,N}
\]
obviously vanish.
Any nonzero summand of $E_{i+1,p}v_{L,N}$ has a factor $z^m w_j v_{i+1}$ for some $j$ with $i+1 \leq N-1$.
Then Lemma~\ref{lem:vanish} implies
\[
	T(E_{p,i}v \wedge E_{i+1,p}v_{L,N}) \wedge \ket{-mNL} = 0.
\]
Now the proof of (\ref{eq:remain1}) is complete.

Next we prove (\ref{eq:remain2}).
It is enough to show
\begin{equation}
	T^2 J(X_{N-1}^-) (v \wedge v_{L,N-1}) \wedge \ket{-mNL} = T^2 (J(X_{N-1}^-) v \wedge v_{L,N-1}) \wedge \ket{-mNL} \label{eq:J_square}
\end{equation}
and
\begin{equation}
	T^2 \omega_{N-1}' (v \wedge v_{L,N-1}) \wedge \ket{-mNL} = 0. \label{eq:omega_square}
\end{equation}
We show (\ref{eq:J_square}).
Write $v \wedge v_{L,N-1} = \bigwedge P \otimes \bw \otimes \bv$, where
\begin{align*}
	P &= z_1^{m_1} \cdots z_n^{m_n} (z_{n+1} \cdots z_{n+2L})^m, \\
	\bw &= w_{b_1} \otimes \cdots \otimes w_{b_n} \otimes w_1^{\otimes 2} \otimes \cdots \otimes w_L^{\otimes 2}, \\
	\bv &= v_{a_1} \otimes \cdots \otimes v_{a_n} \otimes (v_N \otimes v_{N-1})^{\otimes L}.
\end{align*}
Consider the difference
\begin{equation}
	\begin{split}
		& J(X_{N-1}^-) (v \wedge v_{L,N-1}) - J(X_{N-1}^-) v \wedge v_{L,N-1} \\
		&= -\sum_{k=1}^n \bigwedge (d_k^{(n+2L)} - d_k^{(n)})(P \otimes \bw) \otimes (X_{N-1}^-)_k \bv \\
		& \quad + \dfrac{\hbar}{2} \sum_{k=1}^n \sum_{j=1}^L \left( \bigwedge (P \otimes \bw) s_{k,n+2j-1} \otimes (X_{N-1}^-)_k \bv + \bigwedge (P \otimes \bw) s_{k,n+2j} \otimes (X_{N-1}^-)_k \bv \right) \\
		& \quad \quad + \sum_{j=1}^L \left( \bigwedge (P \otimes \bw) y_{n+2j-1}^{(n+2L)} \otimes (X_{N-1}^-)_{n+2j-1} \bv + \bigwedge (P \otimes \bw) y_{n+2j}^{(n+2L)} \otimes (X_{N-1}^-)_{n+2j} \bv \right). \label{eq:difference}
	\end{split}
\end{equation}
The term
\[
	\sum_{j=1}^L \bigwedge (P \otimes \bw) y_{n+2j-1}^{(n+2L)} \otimes (X_{N-1}^-)_{n+2j-1} \bv
\]
vanishes since $(X_{N-1}^-)_{n+2j-1} \bv = 0$.
Then substituting
\[
	y_{n+2j}^{(n+2L)} = - d_{n+2j}^{(n+2L)} + \dfrac{\hbar}{2}\left( \sum_{k > n+2j} s_{n+2j,k} - \sum_{k < n+2j} s_{k,n+2j} \right)
\]
to (\ref{eq:difference}), we obtain
\begin{equation*}
	\begin{split}
		& J(X_{N-1}^-) (v \wedge v_{L,N-1}) - J(X_{N-1}^-) v \wedge v_{L,N-1} \\
		&= -\sum_{k=1}^n \bigwedge (d_k^{(n+2L)} - d_k^{(n)})(P \otimes \bw) \otimes (X_{N-1}^-)_k \bv \\
		& \ + \dfrac{\hbar}{2} \sum_{k=1}^n \sum_{j=1}^L \bigwedge (P \otimes \bw) s_{k,n+2j-1} \otimes (X_{N-1}^-)_k \bv + \dfrac{\hbar}{2} \sum_{k=1}^n \sum_{j=1}^L \bigwedge (P \otimes \bw) s_{k,n+2j} \otimes (X_{N-1}^-)_k \bv \\
		& \ \ - \sum_{j=1}^L \bigwedge d_{n+2j}^{(n+2L)}(P \otimes \bw)  \otimes (X_{N-1}^-)_{n+2j} \bv \\
		& \ \ \ + \dfrac{\hbar}{2} \sum_{j=1}^L \sum_{k > n+2j} \bigwedge (P \otimes \bw) s_{n+2j,k} \otimes (X_{N-1}^-)_{n+2j} \bv \\
		& \ \ \ \ - \dfrac{\hbar}{2} \sum_{j=1}^L \sum_{k < n+2j} \bigwedge (P \otimes \bw) s_{k,n+2j} \otimes (X_{N-1}^-)_{n+2j} \bv. 
	\end{split}
\end{equation*}

\begin{lem}\label{lem:vanishing_terms}
We have the following:
\begin{enumerate}
\item
\[
	\bigwedge (d_k^{(n+2L)} - d_k^{(n)})(P \otimes \bw) \otimes (X_{N-1}^-)_k \bv = 0
\]
for $k=1, \ldots, n$,
\item
\[
	\bigwedge (P \otimes \bw) s_{k,n+2j} \otimes (X_{N-1}^-)_k \bv = 0
\]
for $k=1, \ldots, n$ and $j=1, \ldots, L$,
\item
\[
	\bigwedge d_{n+2j}^{(n+2L)}(P \otimes \bw)  \otimes (X_{N-1}^-)_{n+2j} \bv = 0
\]
for $j=1, \ldots, L$,
\item
\[
	\bigwedge (P \otimes \bw) s_{n+2j,k} \otimes (X_{N-1}^-)_{n+2j} \bv = 0
\]
for $j=1, \ldots, L$ and $k > n+2j$,
\item
\[
	\bigwedge (P \otimes \bw) s_{k,n+2j} \otimes (X_{N-1}^-)_{n+2j} \bv = 0
\]
if either \begin{itemize}
\item $n < k < n+2j$; or
\item $k=1, \ldots, n$ and $a_k=N$,
\end{itemize}
\item
\[
	T^2 \left( \bigwedge (P \otimes \bw) s_{k,n+2j} \otimes (X_{N-1}^-)_{n+2j} \bv \right) \wedge \ket{-mNL} = 0
\]
if $k=1, \ldots, n$ and $a_k \leq N-2$.
\end{enumerate}
\end{lem}

\begin{proof}
We prove (i).
We have
\[
	d_k^{(n+2L)} - d_k^{(n)} = \hbar \left( 1 + \sum_{j=1}^{2L} \left( \dfrac{z_{n+j}}{z_k - z_{n+j}}(1 - K_{k,n+j})P_{k,n+j} + r_{k,n+j} \right) \right).
\]
By Lemma~\ref{lem:standard}, we have
\begin{equation*}
	\dfrac{z_{n+j}}{z_k - z_{n+j}}(1 - K_{k,n+j}) P = - (P + P'),
\end{equation*}
where
\[
	P'= \sum_{r=1}^{m-m_k-1} z_k^r z_{n+j}^{-r} P,
\]
and
\begin{equation*}
	\sum_{j=1}^{2L} r_{k,n+j} \bw = \bw + \sum_{\substack{1 \leq j \leq L\\ j < b_k}} \left( P_{k,n+2j-1}\bw + P_{k,n+2j}\bw \right).
\end{equation*}
Hence we have
\begin{equation*}
	\begin{split}
		& (d_k^{(n+2L)} - d_k^{(n)})(P \otimes \bw) \\
		&= \hbar \Bigg( 2 P \otimes \bw + \sum_{\substack{1 \leq j \leq L\\ j < b_k}} \Big( P \otimes P_{k,n+2j-1}\bw + P \otimes P_{k,n+2j}\bw \Big) - \sum_{j=1}^{2L} P \otimes P_{k,n+j} \bw - \sum_{j=1}^{2L} P' \otimes P_{k,n+j} \bw \Bigg) \\
		&= \hbar \left( -\sum_{\substack{1 \leq j \leq L\\ j > b_k}} \Big( P \otimes P_{k,n+2j-1}\bw + P \otimes P_{k,n+2j}\bw \Big) - \sum_{j=1}^{2L} P' \otimes P_{k,n+j}\bw \right).
	\end{split}
\end{equation*}
A similar argument as in the proof of Lemma~\ref{lem:Jvanish} (i) implies the assertion.

We prove (ii).
We may assume $a_k=N-1$.
Then $(P \otimes \bw) s_{k,n+2j} \otimes (X_{N-1}^-)_k \bv$ has $z^m w_j v_N$ as its $k$-th and ($n+2j-1$)-th factors.
Hence it vanishes after taking wedge.

We prove (iii).
We calculate $d_{n+2j}^{(n+2L)}(P \otimes \bw)$.
By Lemma~\ref{lem:standard}, we have
\begin{equation*}
	\dfrac{z_{n+2j}}{z_k - z_{n+2j}}(1 - K_{k,n+2j}) P =\begin{cases}
	0 \text{ if } n < k < n+2j, \\
	- (P + P') \text{ if } k=1, \ldots, n,
	\end{cases}
\end{equation*}
where
\[
	P'= \sum_{r=1}^{m-m_k-1} z_k^r z_{n+2j}^{-r} P,
\]
and
\begin{equation*}
	\dfrac{z_{k}}{z_k - z_{n+2j}}(1 - K_{k,n+2j}) P = 0
\end{equation*}
for $k > n+2j$.
Note that for $\bw = w_{b_1} \otimes \cdots \otimes w_{b_n} \otimes w_1^{\otimes 2} \otimes \cdots \otimes w_L^{\otimes 2}$, we see that
\begin{gather*}
	k < n+2j \text{ and } j \leq b_k \text{ imply } 1 \leq k \leq n, \\
	k > n+ 2j \text{ implies } j \geq b_k.
\end{gather*}
Hence we have
\begin{equation*}
	\sum_{k<n+2j} r_{k,n+2j} \bw = \dfrac{1}{2} \sum_{\substack{1 \leq k \leq n\\ j=b_k}}\bw + \sum_{\substack{1 \leq k \leq n\\ j < b_k}} P_{k,n+2j}\bw
\end{equation*}
and
\[
	\sum_{k>n+2j} r_{n+2j,k} \bw = 0.
\]
Therefore
\begin{multline*}
		d_{n+2j}^{(n+2L)}(P \otimes \bw) = \text{(constant)} P \otimes \bw \\
		+ \hbar \left( - \sum_{k=1}^n (P+P') \otimes P_{k,n+2j} \bw + \dfrac{1}{2} \sum_{\substack{1 \leq k \leq n\\ j = b_k}} P \otimes \bw + \sum_{\substack{1 \leq k \leq n\\ j < b_k}} P \otimes P_{k,n+2j}\bw \right).
\end{multline*}
We have
\[
	\bigwedge P \otimes \bw \otimes (X_{N-1}^-)_{n+2j} \bv = 0
\]
since $P \otimes \bw \otimes (X_{N-1}^-)_{n+2j} \bv$ has $z^m w_j v_N$ as its ($n+2j-1$)-th and ($n+2j$)-th factors.
We have
\[
	\bigwedge P \otimes P_{k,n+2j}\bw \otimes (X_{N-1}^-)_{n+2j} \bv = 0
\]
since $P \otimes P_{k,n+2j}\bw \otimes (X_{N-1}^-)_{n+2j} \bv$ has $z^m w_{b_k} v_N$ as its ($n+2j$)-th and ($n+2b_k-1$)-th factors.
We have
\[
	\bigwedge P' \otimes P_{k,n+2j}\bw \otimes (X_{N-1}^-)_{n+2j} \bv = 0
\]
by the same reason as (\ref{eq:lower_degree}).

We prove (iv).
If $k=n+2p-1$ for some $p$, then $(P \otimes \bw) s_{n+2j,k} \otimes (X_{N-1}^-)_{n+2j} \bv$ has $z^m w_j v_N$ as its ($n+2j-1$)-th and ($n+2p-1$)-th factors.
If $k=n+2p$, then it has $z^m w_p v_N$ as its ($n+2j$)-th and ($n+2p-1$)-th factors.
Hence in both cases it vanishes after taking wedge.

We prove (v).
First suppose $n < k < n+2j$.
Then $(P \otimes \bw) s_{k,n+2j} \otimes (X_{N-1}^-)_{n+2j} \bv$ has $z^m w_{b_k} v_N$ as its ($n+2b_k-1$)-th and ($n+2j$)-th factors.
Next suppose $k=1, \ldots, n$ and $a_k=N$.
Then it has $z^m w_j v_N$ as its $k$-th and ($n+2j-1$)-th factors.
Hence in both cases it vanishes after taking wedge.

We prove (vi).
Since $(P \otimes \bw) s_{k,n+2j} \otimes (X_{N-1}^-)_{n+2j} \bv$ has $z^m w_j v_{a_k}$ as its $k$-th factor, Lemma~\ref{lem:vanish} implies the assertion.
\end{proof}

We obtain
\begin{multline*}
	T^2 \left( J(X_{N-1}^-) (v \wedge v_{L,N-1}) - J(X_{N-1}^-) v \wedge v_{L,N-1} \right) \wedge \ket{-mNL} = \dfrac{\hbar}{2} \sum_{\substack{1 \leq k \leq n\\ a_k = N-1}} \sum_{j=1}^L \\ 
	T^2 \left( \bigwedge (P \otimes \bw) s_{k,n+2j-1} \otimes (X_{N-1}^-)_k \bv - \bigwedge (P \otimes \bw) s_{k,n+2j} \otimes (X_{N-1}^-)_{n+2j} \bv \right) \wedge \ket{-mNL}.
\end{multline*}
Thus the equality (\ref{eq:J_square}) follows from the next lemma.

\begin{lem}
Suppose $a_k = N-1$.
Then we have
\[
	\bigwedge (P \otimes \bw) s_{k,n+2j-1} \otimes (X_{N-1}^-)_k \bv = \bigwedge (P \otimes \bw) s_{k,n+2j} \otimes (X_{N-1}^-)_{n+2j} \bv
\]
for $k=1, \ldots, n$ and $j=1, \ldots,L$.
\end{lem}

\begin{proof}
Suppose $a_k=N-1$.
Then $(P \otimes \bw) s_{k,n+2j-1} \otimes (X_{N-1}^-)_k \bv$ has
\[
	z^m w_j v_N,\ z^{m_k} w_{b_k} v_N,\ z^m w_j v_{N-1}
\]
as its $k$-th, ($n+2j-1$)-th, ($n+2j$)-th factors, while $(P \otimes \bw) s_{k,n+2j} \otimes (X_{N-1}^-)_{n+2j} \bv$ has
\[
	z^m w_j v_{N-1},\ z^m w_j v_N,\ z^{m_k} w_{b_k} v_N
\]
as its $k$-th, ($n+2j-1$)-th, ($n+2j$)-th factors.
Their remaining factors are the same.
Hence they coincide after taking wedge.
\end{proof}

We show (\ref{eq:omega_square}).
Note that
\begin{equation*}
	\omega_{N-1} = - \left( \sum_{p=1}^{N-2} \left( E_{N,p} E_{p,N-1} + E_{p,N-1} E_{N,p} \right) + H_{N-1} E_{N,N-1} + E_{N,N-1} H_{N-1} \right)
\end{equation*}
In $\omega_{N-1}'(v \wedge v_{L,N-1})$, all terms but concerning
\[
	E_{N,p}v \wedge E_{p,N-1}v_{L,N-1}
\]
obviously vanish.
Any nonzero summand of $E_{p,N-1}v_{L,N-1}$ has a factor $z^m w_j v_N \wedge z^m w_j v_p$ for some $j$ with $p \leq N-2$.
Then Lemma~\ref{lem:vanish} implies
\[
	T^2(E_{N,p}v \wedge E_{p,N-1}v_{L,N-1}) \wedge \ket{-mNL} = 0.
\]
Now the proof of (\ref{eq:remain2}) is complete.

\def\cprime{$'$} \def\cprime{$'$} \def\cprime{$'$} \def\cprime{$'$}
\providecommand{\bysame}{\leavevmode\hbox to3em{\hrulefill}\thinspace}
\providecommand{\MR}{\relax\ifhmode\unskip\space\fi MR }
\providecommand{\MRhref}[2]{%
  \href{http://www.ams.org/mathscinet-getitem?mr=#1}{#2}
}
\providecommand{\href}[2]{#2}


\begin{thebibliography}{FFNR}

\bibitem[D]{MR831053}
Vladimir~G. Drinfel{\cprime}d, \emph{Degenerate affine {H}ecke algebras and
  {Y}angians}, Functional Anal. Appl. \textbf{20} (1986), no.~1, 62--64.

\bibitem[FFNR]{MR2827177}
Boris Feigin, Michael Finkelberg, Andrei Negut, and Leonid Rybnikov,
  \emph{Yangians and cohomology rings of {L}aumon spaces}, Selecta Math. (N.S.)
  \textbf{17} (2011), no.~3, 573--607.

\bibitem[G1]{MR2199856}
Nicolas Guay, \emph{Cherednik algebras and {Y}angians}, Int. Math. Res. Not.
  (2005), no.~57, 3551--3593.

\bibitem[G2]{MR2323534}
\bysame, \emph{Affine {Y}angians and deformed double current algebras in type
  {A}}, Adv. Math. \textbf{211} (2007), no.~2, 436--484.

\bibitem[K]{kodera}
Ryosuke Kodera, \emph{Affine {Y}angian action on the {F}ock space}, preprint
  arXiv:1506.01246, 2015.

\bibitem[L]{MR3202706}
Bernard Leclerc, \emph{Fock space representations of
  {$U_q(\widehat{\mathfrak{sl}}_n)$}}, Geometric methods in representation
  theory. {I}, S\'emin. Congr., vol.~24, Soc. Math. France, Paris, 2012,
  pp.~343--385.

\bibitem[STU]{MR1603798}
Yoshihisa Saito, Kouichi Takemura, and Denis Uglov, \emph{Toroidal actions on
  level {$1$} modules of {$U_q(\widehat{\mathfrak{sl}}_n)$}}, Transform. Groups
  \textbf{3} (1998), no.~1, 75--102.

\bibitem[TU1]{MR1600311}
Kouichi Takemura and Denis Uglov, \emph{Level-{$0$} action of
  {$U_q(\widehat{\mathfrak{sl}}_n)$} on the {$q$}-deformed {F}ock spaces},
  Comm. Math. Phys. \textbf{190} (1998), no.~3, 549--583.

\bibitem[TU2]{MR1710750}
\bysame, \emph{Representations of the quantum toroidal algebra on highest
  weight modules of the quantum affine algebra of type {$\mathfrak{gl}_N$}},
  Publ. Res. Inst. Math. Sci. \textbf{35} (1999), no.~3, 407--450.

\bibitem[U1]{MR1724950}
Denis Uglov, \emph{Symmetric functions and the {Y}angian decomposition of the
  {F}ock and basic modules of the affine {L}ie algebra
  {$\hat{\mathfrak{sl}}_N$}}, Quantum many-body problems and representation
  theory, MSJ Mem., vol.~1, Math. Soc. Japan, Tokyo, 1998, pp.~183--241.

\bibitem[U2]{uglov}
\bysame, \emph{Yangian actions on higher level irreducible integrable modules
  of $\widehat{\mathfrak{gl}}_n$}, arXiv:9802048, 1998.

\bibitem[U3]{MR1768086}
\bysame, \emph{Canonical bases of higher-level {$q$}-deformed {F}ock spaces and
  {K}azhdan-{L}usztig polynomials}, Physical combinatorics ({K}yoto, 1999),
  Progr. Math., vol. 191, Birkh\"auser Boston, Boston, MA, 2000, pp.~249--299.

\bibitem[V]{MR1818101}
Michela Varagnolo, \emph{Quiver varieties and {Y}angians}, Lett. Math. Phys.
  \textbf{53} (2000), no.~4, 273--283.

\bibitem[VV]{MR1626481}
Michela Varagnolo and Eric Vasserot, \emph{Double-loop algebras and the {F}ock
  space}, Invent. Math. \textbf{133} (1998), no.~1, 133--159.

\end{thebibliography}

\end{document}